\documentclass{article}
\usepackage{hyperref}
\usepackage{xcolor}
\usepackage{amsmath}
\usepackage{amsthm}
\usepackage{amssymb}
\usepackage{enumitem}
\usepackage{indentfirst}
\numberwithin{equation}{section}
\newtheorem{theorem}{Theorem}[section]
\newtheorem{definition}[theorem]{Definition}
\newtheorem{corollary}[theorem]{Corollary}

\newtheorem{lemma}[theorem]{Lemma}
\newtheorem{proposition}[theorem]{Proposition}

\newtheorem{remark}{Remark}[section]
\newtheorem{example}[theorem]{Example}

\newcommand{\set}[1]{\left\{#1\right\}}

\begin{document}

	\title{\bf Anisotropic quasilinear
		elliptic systems involving nonlinearities with negative critical terms}

	\author{Elisandra Gloss$^{1}$\,, Artur Jorge Marinho$^2$ and Kanishka Perera$^3$\\
		{\small $^1$Departamento de Matem\'{a}tica, Universidade Federal da Para\'{i}ba}\\
		{\small 58051-900 Jo\~ao Pessoa-PB, Brazil}\\
		{\small{$^1$Corresponding author: elisandra.gloss@academico.ufpb.br}}\\
		{\small $^{2,3}$Department of Mathematics and Systems Engineering, Florida Institute of Technology}\\
		{\small Melbourne, FL 32901, USA}\\
		{\small $^{2}$amarinho2024@my.fit.edu, $^{3}$kperera@fit.edu}
	}
	\date{}
	
	\maketitle
	
	\begin{abstract}
		In this paper we investigate existence and multiplicity of solutions for a critical growth anisotropic quasilinear
		elliptic system, depending on two real parameters $\lambda,\,\mu$, that is coupled through a subcritical perturbation term and involves a negative critical part. 
		We identify a specific scaling for the system and introduce a parameter $\gamma$ associated with this scaling, which governs the geometry of the corresponding variational functional. We establish three distinct types of multiplicity results, depending on the regimes $\gamma = 1$, $\gamma > 1$, and $\gamma < 1$. The multiple results of multiplicity of solutions are based on a recent abstract critical point theorem
		for symmetric functionals on product spaces when $\gamma\geq1$ and on a truncation argument for $\gamma<1$.
		
	\end{abstract}
	
	\section{Introduction}
	In this paper we are interested in existence and multiplicity of nontrivial solutions for the class of critical growth anisotropic quasilinear elliptic systems
	\begin{equation}\label{Sys}
		\left\{
		\begin{aligned}
			-\Delta_p u &= \lambda H_u(u,v) - \mu|u|^{p^\ast-2}u,& \text{in } \Omega\\
			-\Delta_q v &= \lambda H_v(u,v) + |v|^{q^\ast-2}v,&\text{in } \Omega\\
			u=v&=0,&\text{on } \partial\Omega
		\end{aligned}
		\right.
		\tag{$P_{\lambda,\mu}$}
	\end{equation}
	where $\Omega\subset\mathbb{R}^N$ is a bounded domain, $\Delta_pu= \operatorname{div}(|\nabla u|^{p-2}\nabla u)$ is the $p-$Laplacian of $u$, $1<p\leq q<N$, $p^\ast=Np/(N-p)$ and $q^\ast=Nq/(N-q)$ are the critical Sobolev exponents, $\lambda, \mu > 0$ are real parameters, and $H\in C^1(\mathbb{R}^2,\mathbb{R})$ is an even function, with subcritical growth, satisfying 
	\[
	H(t^{1/p}u,t^{1/q}v)=t^\gamma H(u,v), \,\,\forall t\geq0,\quad \forall (u,v)\in\mathbb{R}^2,
	\]
	among some additional conditions, depending on whether the parameter $\gamma$ is less than or greater than $1$.

	Equations involving the $p$-Laplacian operator can be used to model a variety of problems, such as chemotaxi \cite{Bendahmane}, predator-prey models \cite{Songzhi}, non-Newtonian fluid dynamics \cite{GilberlandioJ}, \cite{JianWang}. System \eqref{Sys}  can be used to model a stationary state of two non-Newtonian fluids with different viscosities (if $p\neq q$) flowing through a porous medium, where the operators $-\Delta_p u$
	and $-\Delta_q v$ govern the momentum diffusion.  The coupling term $\lambda H_u(u,v)$
	(resp. $\lambda H_v(u,v)$) models the nonlinear interaction between the two fluid phases, while the term $\mu|u|^{p^*-2}u$ represents a critical-growth absorption by the medium and $|v|^{q^*-2}v$ acts as a source for the second component.

	Although there is a large literature associated with critical systems involving the $p-$Laplacian operator,  we are not aware of any results on existence or multiplicity of solutions for problem \eqref{Sys} with a negative critical term if $p\neq q$.
	When $p=q=2$ we refer to \cite{Chen-Zou-CV-2015,Gloss-Medeiros-Severo,Moreno-2025,Silva-Sousa-20},
	where the authors considered systems involving a negative term, in a subcritical or critical setting. Solutions were obtained using the Mountain Pass theorem or minimization on the Nehari manifold, and these arguments require  $\lambda\in(0,\lambda^*)$ for some $\lambda^*>0$. 
	For $p=q$, we can cite \cite{Chu-Tang-2013,Chu-Lei-Suo-2017,Liu-Zhao-Liu-2018}. In \cite{Chu-Tang-2013}  a critical system is considered involving concave-convex nonlinearities with some weight functions that may change sign, but have a nontrivial positive part. In \cite{Chu-Lei-Suo-2017} the authors  studied  a critical problem with lower-order negative perturbation. In \cite{Liu-Zhao-Liu-2018} the authors considered a system with $k$ equations coupled via a critical term with mixed sign. Using a truncation argument, they proved existence of infinitely many sign-changing solutions when $N>p+p^2$.
	If $k=2$, their nonlinearity is 
	$$
	F(u,v)=\frac{1}{p}(\alpha_1|u|^p+\alpha_2|v|^p)+\frac{1}{p^*}(\beta_{1,1}|u|^{p^*}+2\beta_{1,2}|u|^\frac{p^*}{2}|v|^\frac{p^*}{2}+\beta_{2,2}|v|^{p^*}),
	$$ 
	with $\beta_{1,2}=\beta_{2,1}\leq0$, $\beta_{i,i}>0$ and $0<\alpha_i<\lambda_1(\Omega)$ for $i=1,2$, where $\lambda_1(\Omega)$ is the first eigenvalue of $-\Delta_p$. Even in the case $p=q$, our system \eqref{Sys} exhibits a different structure from that in \cite{Liu-Zhao-Liu-2018}, since it involves a general subcritical term $H$, 
	which allows us to treat the problem for parameters beyond the first eigenvalue,
	and addresses the case $\beta_{1,1}=-\mu<0$.
	
	The present work is motivated by the recent results in \cite{Marinho-Perera-2025}, where critical anisotropic systems with $p\neq q$ were studied. Considering $\mu=-1$ and $H(u,v)=|u|^a|v|^b$ for $a,b>1$ such that $\frac{a}{p_0}+\frac{b}{q_0}=1$ for some $p_0\in(1,p^\ast)$ and  $q_0\in(1,q^\ast)$ and analyzing the cases $\gamma>1$, $\gamma=1$ and $\gamma<1$ for $\gamma=\frac{a}{p}+\frac{b}{q}$, existence and multiplicity results were proved.
	A closely related problem was also studied by adding a negative subcritical (``sublinear'') perturbation to the system. Here we complement the results obtained in \cite{Marinho-Perera-2025} since we consider a more general class of functions $H$ and deal with a negative critical term, which leads to additional difficulties in establishing some compactness conditions for the functional associated with system \eqref{Sys}.
	
	\medskip
	
	Throughout this article we will assume that $H$ satisfies  the following assumptions:
	\begin{enumerate}[label={($h_\arabic*$)},  
		ref=($h_\arabic*$)]               
		\item\label{H basic} 
		$H\in C^1(\mathbb{R}^2,\mathbb{R})$ is an even  function satisfying
		\[
		H(u,v)\ge0,\,\forall (u,v)\in\mathbb{R}^2\quad\text{and}\quad H(u,v)>0\quad\text{for}\quad uv\ne0.
		\]
		
		\item\label{H subcritic} 
		There are $C>0$, $p_1\in(1,p^\ast)$, $q_1\in(1,1+q^\ast(p^\ast-1)/p^\ast)$, $p_2\in(1,1+p^\ast(q^\ast-1)/q^\ast)$ and $q_2\in(1,q^\ast)$ such that
		\[
		|H_u(u,v)|\leq C\left(|u|^{p_1-1}+|v|^{q_1-1}+1\right)\quad\text{and} \quad |H_v(u,v)|\leq C\left(|u|^{p_2-1}+|v|^{q_2-1}+1\right),\quad\forall (u,v)\in\mathbb{R}^2.
		\]
		\item\label{H ep Cep}
		There exist $\tilde{p},\bar{p}\in(1,p^\ast]$ and $\tilde{q},\bar{q}\in(1,q^\ast]$ such that, for any $\varepsilon>0$ there is $C_\varepsilon>0$ satisfying
		\[
		|H(u,v)|\leq \varepsilon (|u|^{\tilde p}+|u|^{\bar p})+ C_\varepsilon(|v|^{\tilde q}+|v|^{\bar q}),\quad\forall (u,v)\in\mathbb{R}^2.
		\]
		\item\label{H scaling}
		There exists $\gamma>0$ for which one has
		\[
		H(t^{1/p}u,t^{1/q}v)=t^\gamma H(u,v), \,\,\forall t\geq0,\quad \forall (u,v)\in\mathbb{R}^2.
		\]
	\end{enumerate}
	When assuming $\gamma\geq 1$ in \ref{H scaling}, we also consider the hypotheses:
	\begin{enumerate}[label={($h_\arabic*$)},  
		ref=($h_\arabic*$),start=5]               
		\item\label{H_v ep Cep} There exists $r\in(q,q^\ast)$ such that, for any $\varepsilon>0$ there is $C_\varepsilon>0$ satisfying
		\[
		rH(u,v)-H_v(u,v)v\leq \varepsilon \left(|u|^p+ |v|^{q^\ast}\right)+C_\varepsilon,\quad\forall (u,v)\in\mathbb{R}^2.
		\]
		\item\label{AR}
		$H$ satisfies
		\[
		H(u,v)\leq\frac{1}{p}H_u(u,v)u+\frac{1}{q}H_v(u,v)v , \quad \forall (u,v)\in\mathbb{R}^2.
		\]
	\end{enumerate}
	
	\begin{remark}
		\begin{description}
			\item[i)] 
			There are many functions $H$ satisfying assumptions \ref{H basic}-\ref{H scaling}. A simple example is
			\begin{equation*}
				H(u,v)=\sum_{i=1}^kc_i |u|^{a_i}|v|^{b_i}+d|v|^{q\gamma},
			\end{equation*}
			where $a_i,b_i>1$ are such that $\frac{a_i}{p_i}+\frac{b_i}{q_i}=1$ for some $p_i\in(1,p^\ast)$ and $q_i\in(1,q^\ast)$, $\gamma=\frac{a_i}{p}+\frac{b_i}{q}\,$ and   $c_i,d\geq0$ are such that $(c_1,\cdots,c_k,d)\ne0$. If  $p_i\leq p$ for all $i$ and $\gamma\geq1$, then $H$ also satisfies \ref{H_v ep Cep}-\ref{AR}. When $p=q$, if \ref{H scaling} holds, then $H$ is a $p\gamma$-homogeneous function. We can cite \cite{Barbosa-Montenegro-2011,DeMoraisFilho-Souto-99} for more examples of homogeneous functions.
			\item[ii)] Condition \ref{H subcritic} guarantees that $H$ is a subcritical term, whereas assumptions \ref{H ep Cep} and \ref{H_v ep Cep} are technical conditions used to handle the negative critical term. Observe that the constants in \ref{H ep Cep} may coincide. This assumption  does not allow functions of the form $H(u,v)=c |u|^{p\gamma}+d|v|^{q\gamma}$, with $c\neq0$.
		\end{description}
	\end{remark}
	
	For a function $\widetilde{H}$ satisfying \ref{H basic},\ref{H subcritic} and \ref{H scaling} with $\gamma=1$, Theorem 10.10 from \cite{Pe-Ag-OR-book} can be applied to ensure that
	the nonlinear eigenvalue problem
	\begin{equation*}
		\left\{
		\begin{aligned}
			-\Delta_p u &= \lambda \widetilde{H}_u(u,v), & \text{in } \Omega\\
			-\Delta_q v &= \lambda \widetilde{H}_v(u,v), &\text{in } \Omega\\
			u=v&=0,&\text{on } \partial\Omega
		\end{aligned}
		\right.
	\end{equation*}
	has a sequence of positive eigenvalues $\{\lambda_k\}_{k\in \mathbb{N}}$ such that $\lambda_k\nearrow\infty$ (see Theorem \ref{lambdak}). Considering this sequence, we will prove the following results. 
	
	\begin{theorem} \label{th 1 =1}
		Assume that $H$ satisfies \ref{H basic}-\ref{AR}, with $\gamma=1$ in \ref{H scaling}. Then,
		for $k,m\geq1$ satisfying $\lambda_{k}=\cdots=\lambda_{k+m-1}<\lambda_{k+m}$,
		there exist $\delta_k, \mu_k>0$ such that for all $\lambda\in(\lambda_k-\delta_k,\lambda_k)$ and $\mu\in(0,\mu_k)$ the system \eqref{Sys} has $m$ pairs of  nontrivial weak solutions $\pm w_j^{\lambda,\mu}$, $j=1,\cdots,m$, at positive energy levels, such that $w_j^{\lambda,\mu}\to0$ as $\lambda\to\lambda_k^-$ and $\mu\to0^+$.
	\end{theorem}
	
	Since  $\lambda_k\nearrow\infty$,  the following result is an immediate consequence of Theorem \ref{th 1 =1}.
	
	\begin{corollary}
		If $H$ satisfies \ref{H basic}-\ref{AR}, with $\gamma=1$ in \ref{H scaling}, then for each $k\geq1$, there exist $\delta_k, \mu_k>0$  such that the system \eqref{Sys} has a pair of nontrivial solutions for all $\lambda\in(\lambda_k-\delta_k,\lambda_k)$ and $\mu\in(0,\mu_k)$.
	\end{corollary}
	
	When $\gamma>1$, we have the following multiplicity result for large values of $\lambda$ and small values of $\mu>0$.
	\begin{theorem} \label{Th2 >1}
		If $H$ satisfies \ref{H basic}-\ref{AR}, with $\gamma>1$ in \ref{H scaling}, then
		for any $m\geq1$, 
		there exists $\Lambda_m^*>0$  such that for each $\lambda\ge \Lambda_m^*$, there exists $\mu_\lambda>0$ such that for all $\mu\in(0,\mu_\lambda)$, the system \eqref{Sys} has $m$ pairs of  nontrivial weak solutions $\pm w_j^{\lambda,\mu}$, $j=1,\cdots,m$ at positive energy levels. In particular, the number of solutions goes to infinity as $\lambda\to\infty$.
	\end{theorem}
	
	On the other hand, for $\gamma<1$ we obtain infinitely many solutions for $\lambda,\mu>0$ small.
	\begin{theorem} \label{Th3 <1}
		Assume that $H$ satisfies \ref{H basic}-\ref{H scaling} with $\gamma<1$. Then
		there exist  $\Lambda^*,\,\mu^*>0$  such that for all $\lambda\in(0,\Lambda^*)$ and $\mu\in(0,\mu^*)$, the system \eqref{Sys} has a sequence of  weak solutions $\{w_j^{\lambda,\mu}\}_{j\in\mathbb{N}}$ at negative energy levels such that $w_j^{\lambda,\mu}\to0$ as $\lambda\to0^+$.
	\end{theorem}
	
	Considering $H(u,v)=|u|^a|v|^b$, with $a,b>1$ satisfying $\frac{a}{p_0}+\frac{b}{q_0}=1$ for some  $p_0\in(1,p^*)$ and $q_0\in(1,q^*)$, we see that $H$ satisfies \ref{H basic}-\ref{H scaling} with $\gamma=\frac{a}{p}+\frac{b}{q}$. Moreover, it satisfies \ref{H_v ep Cep} if $p_0\leq p$ and \ref{AR} if $\gamma\geq1$. In particular,
	\[
	\begin{cases}
		p_0=p\,\,\text{and}\,\,q_0=q\quad\Rightarrow &\gamma=1\\
		p_0\geq p\,\,\text{and}\,\,q_0>q\quad\Rightarrow &\gamma>1\\
		p_0\leq p\,\,\text{and}\,\,q_0<q\quad\Rightarrow &\gamma<1,
	\end{cases}
	\]
	so we can apply our theorems to ensure the existence and multiplicity of solutions to system \eqref{Sys} with this model function. 
	
	Proofs of our multiplicity results will be based on a recent abstract critical point theorem for symmetric functionals on product spaces (see \cite[Theorem 2.5]{Marinho-Perera-2025}) and a classical result of critical point theory (see, e.g., \cite[Proposition 3.36]{Pe-Ag-OR-book}).
	
	The paper is organized as follows. In Section \ref{Sec prel}, we present the variational framework that will be used throughout this work, along with the abstract results underlying our proofs, including the one concerning the existence of nonlinear eigenvalues. In Section \ref{Sec CC}, we establish a compactness condition for the functional associated with Problem \eqref{Sys}. Section \ref{Sec proofs} is devoted to the proofs of our main results.

	\section{Preliminary results}\label{Sec prel}
	
	We will work in the product space  $W=W^{1,p}_0(\Omega)\times W^{1,q}_0(\Omega)$ endowed with the norm
	\[
	\|w\|=\|\nabla u\|_p+\|\nabla v\|_q,\quad w=(u,v)\in W.
	\]
	Since $H\in C^1(\mathbb{R}^2)$ satisfies \ref{H subcritic}, there are $p_0\in(1,p^\ast)$, $q_0\in(1,q^\ast)$ and a positive constant $M_H$  satisfying
	\begin{equation}\label{H MH}
		|H(u,v)|\leq M_H(1+|u|^{p_0}+ |v|^{q_0}),\quad\forall (u,v)\in\mathbb{R}^2.
	\end{equation}
	Then one can see that the functional $E:W\to\mathbb{R}$ associated to problem \eqref{Sys}, given by
	\begin{equation*}
		E(u,v)=\frac{1}{p}\int_\Omega|\nabla u|^p\,dx+\frac{1}{q}\int_\Omega|\nabla v|^q\,dx -\lambda\int_\Omega H(u,v)\,dx+\frac{\mu}{p^{\ast}}\int_\Omega|u|^{p^{\ast}}\,dx-\frac{1}{q^{\ast}}\int_\Omega|v|^{q^{\ast}}\,dx,
	\end{equation*}
	is of class $C^1$, with derivative
	\begin{equation}\label{E'}
		\begin{array}{ccl}
			E'(u,v)(\varphi,\psi)&=&\displaystyle\int_\Omega\left(|\nabla u|^{p-2}\nabla u\nabla \varphi+|\nabla v|^{q-2}\nabla v\nabla \psi\right)\,dx -\lambda\int_\Omega \left(H_u(u,v)\varphi+H_v(u,v)\psi\right)\,dx\\
			&&\displaystyle+\mu\int_\Omega|u|^{p^{\ast}-2}u\varphi\,dx-\int_\Omega|v|^{q^{\ast}-2}v\psi\,dx.
		\end{array}
	\end{equation}
	So critical points of $E$ are weak solutions of \eqref{Sys}.
	
	\subsection{Linking theorems}
	
	We will prove Theorem \ref{th 1 =1} and Theorem \ref{Th2 >1} using two abstract results proved in Theorem 2.5 and Corollary 2.6 in \cite{Marinho-Perera-2025} that are based on the $\mathbb{Z}_2$-cohomological index of Fadell and Rabinowitz
	\cite{Fadell-Rabinowitz-1978}.  
	\begin{definition}[see \cite{Fadell-Rabinowitz-1978}]\label{Definition 301}
		Let $\mathcal{A}$ denote the class of symmetric subsets of $W \setminus \set{0}$. For $A \in \mathcal{A}$, let $\overline{A} = A/\mathbb{Z}_2$ be the quotient space of $A$ with each $u$ and $-u$ identified, let $f : \overline{A} \to \mathbb{R}\text{P}^\infty$ be the classifying map of $\overline{A}$, and let $f^\ast : H^\ast(\mathbb{R}\text{P}^\infty) \to H^\ast(\overline{A})$ be the induced homomorphism of the Alexander-Spanier cohomology rings. The cohomological index of $A$ is defined by
		\[
		i(A) = \begin{cases}
			0 & \text{if } A = \emptyset\\[5pt]
			\sup \set{m \ge 1 : f^\ast(\omega^{m-1}) \ne 0} & \text{if } A \ne \emptyset,
		\end{cases}
		\]
		where $\omega \in H^1(\mathbb{R}\text{P}^\infty)$ is the generator of the polynomial ring $H^\ast(\mathbb{R}\text{P}^\infty) = \mathbb{Z}_2[\omega]$.
	\end{definition}
	
	\begin{example}
		The classifying map of the unit sphere $S^N$ in $\mathbb{R}^{N+1},\, N \ge 0$ is the inclusion $\mathbb{R}\text{P}^N \hookrightarrow \mathbb{R}\text{P}^\infty$, which induces isomorphisms on the cohomology groups $H^l$ for $l \le N$, so $i(S^N) = N + 1$.
	\end{example}
	
	Before stating the abstract results, we need to establish some notations and definitions. Let $I:W\to\mathbb{R}$ be given by
	\begin{equation}\label{I defi}
		I(u,v)=\frac{1}{p}\int_\Omega|\nabla u|^p\,dx+\frac{1}{q}\int_\Omega|\nabla v|^q\,dx,
	\end{equation}
	and define
	\begin{equation}\label{M def}
		\mathcal{M} = \{ w \in W \setminus \{0\} : I(w) = 1 \}.
	\end{equation}
	Considering the scaling
	\begin{equation}\label{scaling def}
		[0,\infty)\times W\to W, \quad (t,w)\mapsto w_t:=(t^{1/p}u,t^{1/q}v),
	\end{equation}
	we see that
	\[
	I(w_t)=tI(w),\quad\forall t\geq0\,\,\text{and}\,\,w\in W.
	\]
	The radial projection on $\mathcal{M}$, $\pi: W\setminus\{0\}\to \mathcal{M}$, is  given by $\pi(w)=w_{t_w}$ where $t_w=I(w)^{-1}$. For $\rho > 0$, we define
	\begin{equation*}
		\mathcal{M}_\rho = \left\{ w \in W \setminus \{0\}:I(w) = \rho \right\} = \{ w_\rho : w \in \mathcal{M} \}.
	\end{equation*}
	Let $E \in C^1(W,\mathbb{R})$ be an even functional, that is,
	\[
	E(-u) = E(u) \quad \text{for all } u \in W.
	\]
	Assume that there exists $c^\ast > 0$ such that $E$ satisfies the $(PS)_c$ condition for every $c \in (0,c^\ast)$. Let $\Gamma$ denote the group of odd homeomorphisms of $W$ that coincide with the identity outside $E^{-1}(0,c^\ast)$. Let $\mathcal{A}^\ast$ be the class of symmetric subsets of $W$, i.e., subsets $M \subset W$ such that if $u \in M$, then $-u \in M$. According to Benci \cite{V. Benci}, for any $M \in \mathcal{A}^*$, its pseudo-index $i^*(M)$ with respect to $i$, $\mathcal{M}_\rho$, and $\Gamma$ is defined by
	\[
	i^*(M) = \min_{\gamma \in \Gamma} i(\gamma(M) \cap \mathcal{M}_\rho).
	\]
	
	\begin{theorem}[\cite{Marinho-Perera-2025}, Theorem 2.5]\label{Theorem 2.5 MP} Consider $E \in C^1(W,\mathbb{R})$ an even functional and suppose that $E$ satisfies the $(PS)_c$ condition for all $c\in(0,c^*)$, for some $c^*>0$.
		Let $A_0$ and $B_0$ be symmetric subsets of $\mathcal{M}$ such that $A_0$ is compact, $B_0$ is closed, and
		\begin{equation*}
			i(A_0) \geq k + m - 1, \qquad i(\mathcal{M} \setminus B_0) \leq k - 1
		\end{equation*}
		for some $k, m \geq 1$. Let $R > \rho > 0$ and let
		\[
		X = \{w_t : w \in A_0,\, 0 \leq t \leq R \}, \quad
		A = \{w_R : w \in A_0 \}, \quad
		B = \{w_\rho : w \in B_0 \}.
		\]
		Assume that
		\begin{equation*}
			\sup_{w \in A} E(w) \leq 0 < \inf_{w \in B} E(w), \qquad
			\sup_{w \in X} E(w) < c^*.
		\end{equation*}
		For $j = k, \ldots, k + m - 1$, let
		\(
		\mathcal{A}^*_j = \{M \in \mathcal{A}^* : M \text{ is compact and } i^*(M) \geq j \}
		\)
		and set
		\[
		c^*_j := \inf_{M \in \mathcal{A}^*_j} \max_{w \in M} E(w).
		\]
		Then 
		$$
		\inf_{w\in B}E(w) \leq c^*_k \leq \cdots \leq c^*_{k+m-1} \leq \sup_{w\in X}E(w),
		$$
		each $c^*_j$ is a critical value of $E$, and $E$ has $m$ distinct pairs of associated critical points.
	\end{theorem}
	When $k=1$, one can take $B_0=\mathcal{M}$. This gives the following corollary.
	\begin{corollary}[\cite{Marinho-Perera-2025}, Corollary 2.6]\label{Cor-2.6}
		Let $A_0$  be a compact symmetric   subset of $\mathcal{M}$ with
		\(
		i(A_0) =m\geq1.
		\)
		Let $R > \rho > 0$ and let
		\[
		A = \{w_R : w \in A_0 \},\quad
		X = \{w_t : w \in A_0,\, 0 \leq t \leq R \}.
		\]
		Assume that
		\begin{equation*}
			\sup_{w \in A} E(w) \leq 0 < \inf_{w \in \mathcal{M}_\rho} E(w), \qquad
			\sup_{w \in X} E(w) < c^*.
		\end{equation*}
		Then $E$ has $m$ distinct pairs of  critical points $\pm w_j$, $j=1,\cdots,m$ such that
		$$
		\inf_{w\in \mathcal{M}_\rho}E(w) \leq E(w_j) \leq \sup_{w\in X}E(w)
		$$
		for each $j$.
	\end{corollary}
	
	To prove the existence of solutions to the system \eqref{Sys} when $\gamma<1$, more specifically, to prove Theorem \ref{Th3 <1}, we will use the following classical result from critical point theory (see, e.g., \cite[Proposition 3.36]{Pe-Ag-OR-book}).
	
	\begin{proposition}\label{critical point result}
		Let $E$ be an even $C^1$-functional on a Banach space $W$ such that $E(0) = 0$
		and $E$ satisfies the $(PS)_c$ condition for all $c<0$. Let $\mathcal{F}$ denote the class of symmetric subsets
		of $W \backslash\{0\}$. For $k\geq1$, let
		$$
		\mathcal{F}_k = \{M \in \mathcal{F}: i(M) \geq k\}
		$$
		and set
		\begin{equation}\label{ck gamma<1}
			c_k:=\inf_{M\in \mathcal{F}_k}\sup_{u\in M}E(u).
		\end{equation}
		If there exists a $k_0 \geq1$ such that $-\infty<c_k <0$ for all $k \geq k_0$, then $c_{k_0}\leq c_{k_0+1}\leq\cdots  \to0$ is a sequence of critical values of $E$.
	\end{proposition}
	
	\subsection{Nonlinear eigenvalue problem}
	
	Now, for a function $\widetilde{H}$ that satisfies \ref{H basic},\ref{H subcritic} and \ref{H scaling} with $\gamma=1$, let us consider the nonlinear eigenvalue problem
	\begin{equation}\label{NE}
		\left\{
		\begin{aligned}
			-\Delta_p u &= \lambda \widetilde{H}_u(u,v), & \text{in } \Omega\\
			-\Delta_q v &= \lambda \widetilde{H}_v(u,v), &\text{in } \Omega\\
			u=v&=0,&\text{on } \partial\Omega.
		\end{aligned}
		\right.
		\tag{NE}
	\end{equation}
	Recalling \eqref{M def}, let
	$$
	\mathcal{M}^+:=\left\{w=(u,v)\in \mathcal{M}: \int_\Omega \widetilde{H}(u,v)\,dx>0\right\}.
	$$
	Condition \ref{H basic} ensures that $\mathcal{M^+}\neq\emptyset$.
	\bigskip
	Now we define $\Psi:\mathcal{M}^+\to\mathbb{R}$ by
	$$
	\Psi(u,v)=\frac{1}{\int_\Omega \widetilde{H}(u,v)\, dx}.
	$$ 
	Let $\mathcal{F}$ denote the class of symmetric subsets of $\mathcal{M}^+$ and let $i(Y)$ denote the cohomological index of $Y \in\mathcal F$ (see \cite{Fadell-Rabinowitz-1978}). For $k\in\mathbb{N}$, we consider
	$$
	\mathcal{F}_k = \{Y \in \mathcal{F}:  i(Y) \geq k\}
	$$
	and set
	$$
	\lambda_k := \inf_{Y\in\mathcal{F}_k}\sup_{u\in Y}\Psi(u).
	$$
	The next theorem is an immediate consequence of an abstract result proved in \cite{Pe-Ag-OR-book}. It ensures that $(\lambda_k)$ is indeed a sequence of \emph{nonlinear eigenvalues} for \eqref{NE} and provides some additional information as stated below.
	
	\begin{theorem}[\cite{Pe-Ag-OR-book}, Theorem 10.10]\label{lambdak}
		Assume that $1<p, q<N$ and $\widetilde{H}$ satisfies \ref{H basic}, \ref{H subcritic} and \ref{H scaling} with $\gamma=1$. Then $\{\lambda_k\}_{k\in \mathbb{N}}$ is a nondecreasing sequence of positive eigenvalues of \eqref{NE} and $\lambda_k\to \infty$. Moreover,
		\begin{enumerate}[label={(\roman*)},  
			ref=(\roman*)]               
			\item\label{lambdak(i)} The first eigenvalue is  given by
			$$
			\lambda_1=\min_{w\in\mathcal{M}^+}{\Psi}(w)>0.
			$$
			\item\label{lambdak(ii)}  If $\lambda_k = \cdots = \lambda_{k+m-1} = \lambda$, then $i(E_\lambda)\geq m$, where $E_\lambda$ is the set of eigenfunctions associated with $\lambda$ that lie on $\mathcal{M}^+$. 
			\item\label{lambdak(iii)} If $\lambda_k<\lambda<\lambda_{k+1}$, then
			$$i({\Psi}^{\lambda_k} ) = i(\mathcal{M}^+\backslash{\Psi}_{\lambda}) = i({\Psi}^{\lambda} ) 
			= i(\mathcal{M}^+\backslash {\Psi}_{\lambda_{k+1}}) = k,$$
			where 
			${\Psi}^a=\{ w\in \mathcal{M}^+ : {\Psi}(w) \leq a\}$ and 
			${\Psi}_a = \{ w\in \mathcal{M}^+ : {\Psi}(w) \geq a\}$ for $a\in\mathbb{R}$. 
		\end{enumerate}
	\end{theorem}
	
	\section{A compactness condition}\label{Sec CC}
	
	Before showing that the functional $E$ associated with problem \eqref{Sys} satisfies the $(PS)_c$ condition for $c$ bellow a suitable threshold level, let us prove this basic result for $H$.  
	
	\begin{lemma} \label{Lemma H}
		Assume that $H$ satisfies \ref{H basic} and \ref{H subcritic}.  
		\begin{enumerate}
			[label={(\roman*)},  
			ref=(\roman*)]               
			\item \label{Lemma H p star} 
			If \ref{H ep Cep} holds, then given $\varepsilon > 0$, there exists $C=C(\varepsilon,p,\tilde{p},\bar{p},q,\tilde{q},\bar{q},N,\Omega) > 0$ such that:
			\[
			\int_\Omega H(u,v)\,dx \leq \varepsilon\left(1+\int_\Omega|u|^{p^{\ast}}\,dx\right)+ C \left(\int_\Omega|v|^{q^{\ast}}\,dx\right)^{\tilde{q}/q^\ast}+C \left(\int_\Omega|v|^{q^{\ast}}\,dx\right)^{\bar{q}/q^\ast},\quad\forall (u, v) \in W.
			\]
			\item \label{Lemma H < I(w)} 
			If \ref{H scaling} holds, then there exists $\overline{C}=\overline{C}(\Omega,p,{p_1},p_2,q,q_1,q_2,N) > 0$ such that:
			\[
			\int_\Omega H(u,v)\,dx \leq \overline{C}\, [I(u,v)]^\gamma,\quad\forall (u, v) \in W.
			\]
			\item \label{Lemma H grad}
			If \ref{H_v ep Cep} holds, then given $\varepsilon > 0$, there exists $C=C(\varepsilon,p,q,N,\Omega) > 0$ such that:
			\[
			\int_\Omega \left[r H(u,v)-H_v(u,v)v\right]\,dx \leq \varepsilon \int_\Omega|\nabla u|^p\,dx + \varepsilon \int_\Omega|v|^{q^{\ast}}\,dx+C,\quad\forall (u, v) \in W.
			\]
			
		\end{enumerate}
	\end{lemma}
	
	\begin{proof}
		\noindent Proof of \ref{Lemma H p star}: Given $\varepsilon>0$, using \ref{H ep Cep} we get 
		\begin{align*}
			|H(s,t)|\le \frac{\varepsilon}{2(|\Omega|^{1-\frac{\tilde p}{p^\ast}}+|\Omega|^{1-\frac{\bar p}{p^\ast}})}(|s|^{\tilde{p}}+|s|^{\bar{p}})+ \widetilde{C}(|t|^{\tilde q}+|t|^{\bar q})
			,\quad\forall s,t\in\mathbb{R},
		\end{align*}
		where $\widetilde{C}=\widetilde{C}(\varepsilon,|\Omega|,p,\tilde{p},\bar{p},N)$. Using the H\"older's inequality we obtain
		\begin{align*}
			\int_\Omega |H(u,v)|\,dx &\leq \frac{\varepsilon}{2} \left(\int_\Omega|u|^{p^\ast}\,dx \right)^{\tilde{p}/p^\ast}+\frac{\varepsilon}{2} \left(\int_\Omega|u|^{p^\ast}\,dx \right)^{\bar{p}/p^\ast}\\
			&\quad+{C}\left(\int_\Omega|v|^{q^{\ast}}\,dx\right)^{\tilde{q}/q^\ast}+{C}\left(\int_\Omega|v|^{q^{\ast}}\,dx\right)^{\bar{q}/q^\ast},
		\end{align*}
		for all $ (u, v) \in W$, with ${C}={C}(\varepsilon,|\Omega|,p,\tilde{p},\bar{p},q,\tilde{q},\bar{q},N)$. This implies that \ref{Lemma H p star} holds. \\
		
		\noindent Proof of \ref{Lemma H < I(w)}: From \ref{H scaling} we have $H(0,0)=0$. So, using \ref{H subcritic}, there is $C>0$ such that for each $(u,v)\in\mathbb{R}^2$ we get
		\begin{align*}
			|H(u,v)|&=|\nabla H(\theta u,\theta v)\cdot(u,v)|
			\leq C\left(|u|^{p_1}+|u|^{p^*}+|v|^\frac{(q_1-1)p^*}{p^*-1}+|u|^\frac{(p_2-1)q^*}{q^*-1}+|v|^{q^*}+|v|^{q_2}\right)
		\end{align*}
		with $\theta=\theta(u,v)\in(0,1)$. Then, since $W^{1,p}_0(\Omega)\hookrightarrow L^{s}(\Omega)$ for $s\in[1,p^*]$, recalling the definition of $I$ in \eqref{I defi} we obtain
		\begin{align*}
			\int_\Omega |H(u,v)|\,dx
			&\leq \tilde{C}\sum_{j=1}^3\left[\left(\int_\Omega|\nabla u|^p\,dx\right)^{a_j}+\left(\int_\Omega|\nabla v|^q\,dx\right)^{b_j}\right]\\
			&\leq \tilde{C}\sum_{j=1}^3\left[\left(pI(u,v)\right)^{a_j}+\left(qI(u,v)\right)^{b_j}\right],
		\end{align*}
		for all $(u,v)\in W$, for some $a_j,\,b_j>0$, $j=1,2,3$. Then, given $w=(u,v)\in W\setminus\{0\}$ and $w_t=(t^\frac{1}{p}u,t^\frac{1}{q}v)$ (see \eqref{scaling def}) we take $t_w=[I(w)]^{-1}$ so that $I(w_{t_w})=1$.  From \ref{H scaling} we get
		\begin{align*}
			t_w^\gamma \int_\Omega |H(w)|\,dx&=\int_\Omega |H(w_{t_w})|\,dx
			\leq \tilde{C}\sum_{j=1}^3\left[pI(w_{t_w})\right]^{a_j}+\tilde{C}\sum_{j=1}^3\left[qI(w_{t_w})\right]^{b_j}\leq \overline{C},
		\end{align*}
		for a constant $\overline{C}=\overline{C}(\Omega,p,{p_1},p_2,q,q_1,q_2,N)>0$. By the choice of $t_w$ we have the desired result.\\
		
		\noindent Proof of \ref{Lemma H grad}: Similarly, due to Poincar\'e's inequality, there exists $\nu_1=\nu_1(p,N,\Omega)>0$ such that
		\[
		\int_\Omega|u|^p\,dx \le \nu_1\int_\Omega|\nabla u|^p\,dx ,\quad\forall u\in W^{1,p}_0(\Omega).
		\]
		Given $\varepsilon>0$, using \ref{H_v ep Cep}  we get
		\[
		rH(s,t)-H_v(s,t)t\le (\varepsilon/\nu_1)|s|^p+ \varepsilon|t|^{q^\ast}+\widetilde{C}(\varepsilon,\nu_1),\quad\forall s,t\in\mathbb{R}.
		\]
		Thus, 
		\[
		\int_\Omega \left[rH(u,v)-H_v(u,v)v\right]\,dx \leq \varepsilon \int_\Omega|\nabla u|^p\,dx +\varepsilon\int_\Omega|v|^{q^{\ast}}\,dx +\widetilde{C}|\Omega|,\quad\forall (u, v) \in W
		\]
		which proves \ref{Lemma H grad}. 
	\end{proof}
	
	Let $S_q$ denote the best constant for the Sobolev embedding of $W^{1,q}_0(\Omega)$ in $L^{q^\ast}(\Omega)$, i.e.,
	\begin{equation}\label{S def}
		S_q=\inf_{u\in W^{1,q}_0(\Omega) \backslash\{0\}}\frac{\int_\Omega|\nabla v|^q\,dx }{\left(\int_\Omega|v|^{q^{\ast}}\,dx\right)^{q/q^{\ast}}}.
	\end{equation}
	
	\begin{lemma}\label{Lemma PS}
		Assume that $H$ satisfies \ref{H basic}, \ref{H subcritic}, \ref{H_v ep Cep} 
		and \ref{AR}.
		Let $\tilde\lambda, \mu>0$ and let $\lambda\in(0,\tilde\lambda]$.
		Then there exists a positive constant $d_{\tilde\lambda}$ such that the functional $E$ satisfies the  $(PS)_c$ condition for all $c < \frac{1}{N} S_q^{N/q}-\mu d_{\tilde\lambda}$.
	\end{lemma}
	\begin{proof}
		Let $\{w_n\}_n $ be a $(PS)_c$ sequence for $E$ in $W$. We first prove that $w_n= (u_n, v_n)$  is bounded. 
		Since $E(w_n)=c+o(1)$, we have
		\begin{align}\label{E(wn)}
			\frac{1}{p} \int_\Omega |\nabla u_n|^p\, dx+\frac{1}{q} \int_\Omega|\nabla v_n|^q dx -\lambda\int_\Omega H(u_n,v_n)\,dx-\frac{1}{q^{\ast}}\int_\Omega|v_n|^{q^{\ast}}\,dx \nonumber\\
			=-\frac{\mu}{p^{\ast}}\int_\Omega|u_n|^{p^{\ast}}\,dx +c+o(1).
		\end{align}
		Recalling that $\mu > 0$, 
		we get
		\begin{align}\label{E(wn) no mu}
			\frac{1}{p} \int_\Omega |\nabla u_n|^p\, dx+\frac{1}{q} \int_\Omega|\nabla v_n|^q dx  -\lambda\int_\Omega H(u_n,v_n)\,dx-\frac{1}{q^{\ast}}\int_\Omega|v_n|^{q^{\ast}}\,dx \leq c+o(1).
		\end{align}
		Also, once $E'(u_n,v_n)(\varphi,\psi)=o(1)\|(\varphi,\psi)\|$ for any $(\varphi,\psi)\in W$, 
		by \eqref{E'} we have
		\begin{align}
			\int_\Omega\left(|\nabla u_n|^{p-2}\nabla u_n\nabla \varphi+|\nabla v_n|^{q-2}\nabla v_n\nabla \psi\right)\,dx -\lambda\int_\Omega \left(H_s(u_n,v_n)\varphi+H_t(u_n,v_n)\psi\right)\,dx\nonumber\\
			+\mu\int_\Omega|u_n|^{p^{\ast}-2}u_n\varphi\,dx-\int_\Omega|v_n|^{q^{\ast}-2}v_n\psi\,dx=o(1)\|(\varphi,\psi)\|.\label{E'(un,vn)(phi,psi)}
		\end{align}
		For $(\varphi,\psi)=(u_n,0)$ and $(\varphi,\psi)=(0,v_n)$ we obtain
		\begin{align}\label{E'(wn)un}
			\int_\Omega|\nabla u_n|^p\,dx -
			\lambda\int_\Omega H_s(u_n,v_n)u_n\,dx + \mu \int_\Omega|u_n|^{p^{\ast}}\,dx = o(1)\|w_n\|
		\end{align}
		and
		\begin{align}\label{E'(wn)vn}
			\int_\Omega|\nabla v_n|^q\,dx -
			\lambda\int_\Omega H_t(u_n,v_n)v_n\,dx -\int_\Omega|v_n|^{q^{\ast}}\,dx = o(1)\|w_n\|.
		\end{align}
		Now we consider $r\in (q,q^{\ast})$ given in \ref{H_v ep Cep}.
		We  multiply \eqref{E'(wn)vn} by $1/r$ and subtract it from \eqref{E(wn) no mu}, obtaining:
		\begin{align*}
			\frac{1}{p} \int_\Omega|\nabla u_n|^p\,dx &+\left(\frac{1}{q}-\frac{1}{r}\right)  \int_\Omega|\nabla v_n|^q\,dx +
			\left(\frac{1}{r}-\frac{1}{q^{\ast}}\right) \int_\Omega|v_n|^{q^{\ast}}\,dx\nonumber\\
			&\leq \frac{\lambda}{r} \int_\Omega \left[rH(u_n,v_n)-H_t(u_n,v_n)v_n\right]\,dx + c + o(1) + o(1)\|w_n\|.
		\end{align*}
		Considering $\lambda\in(0,\tilde\lambda]$ and $\varepsilon=\theta/(2\tilde\lambda)$, for
		\(
		\theta= \min\{ \frac{1}{r} - \frac{1}{q^{\ast}} \,;\, \frac{1}{q} - \frac{1}{r}\},
		\)
		by Lemma \ref{Lemma H}\ref{Lemma H grad}  we obtain
		\begin{align*}
			\frac{\lambda}{r} \int_\Omega \left[rH(u_n,v_n)-H_t(u_n,v_n)v_n\right]\,dx \leq
			\frac{\theta}{2} \int_\Omega|\nabla u_n|^p\,dx  +\frac{\theta}{2} \int_\Omega| v_n|^{q^{\ast}}\,dx + {C}, 
		\end{align*}
		where ${C}={C}(p,q,r,N,\Omega,\tilde\lambda)$. Hence 
		\begin{align*}
			\frac{\theta}{2} \int_\Omega|\nabla u_n|^p\,dx 
			&+ \frac{\theta}{2}\int_\Omega|\nabla v_n|^q\,dx 
			+ \frac{\theta}{2}\int_\Omega|v_n|^{q^{\ast}}\,dx\leq c + o(1)(1 + \|w_n\|) 
			+ C. 
		\end{align*}
		This implies that
		$\|w_n\|$ is bounded. Moreover,  there exists $K_{\tilde\lambda}=K(p,q,r,N,\Omega,\tilde\lambda)>0$ such that
		\begin{align}\label{unif bdd}
			\limsup_{n\to\infty}\|w_n\|\leq K_{\tilde\lambda}
		\end{align}
		for any $(PS)_c$ sequence $\{w_n\}_n$ of $E$ with $c<S_q^{N/q}/N$ and any $\mu>0$ and  $\lambda\in(0,\tilde\lambda]$.
		Now, taking a subsequence, we may assume $(u_n, v_n) \rightharpoonup (u, v)$   in $W$, $u_n \to u$  in $L^s(\Omega)$ for $s \in [1, p^{\ast})$,  $v_n\to v$ in $L^s(\Omega)$ for $s \in [1, q^{\ast})$, and $u_n(x) \to u(x)$, $v_n(x) \to v(x)$ almost everywhere in $\Omega$. Using \ref{H subcritic} and 
		the general Lebesgue dominated convergence theorem, we see that
		\[
		\int_\Omega H_s(u_n,v_n)\varphi \,dx \to \int_\Omega H_s(u,v)\varphi\,dx\quad \text{and}\quad
		\int_\Omega H_t(u_n,v_n)\psi\,dx \to \int_\Omega H_t(u,v)\psi\,dx,
		\]
		for all $(\varphi,\psi)\in W$, as well as
		\[
		\int_\Omega H_s(u_n,v_n)u_n\,dx \to \int_\Omega H_s(u,v)u\,dx\quad \text{and}\quad
		\int_\Omega H_t(u_n,v_n)v_n\,dx \to \int_\Omega H_t(u,v)v\,dx.
		\]
		Using \eqref{H MH} we also get 
		\[
		\int_\Omega H(u_n,v_n) \,dx \to \int_\Omega H(u,v)\,dx.
		\]
		Moreover, due to the weak convergences $|u_n|^{p^\ast-2}u_n\rightharpoonup |u|^{p^\ast-2}u$ in  $L^\frac{p^\ast}{p^\ast-1}(\Omega)$ and $|v_n|^{q^\ast-2}v_n\rightharpoonup |v|^{q^\ast-2}v$ in  $L^\frac{q^\ast}{q^\ast-1}(\Omega)$ we see that
		\[
		\int_\Omega |u_n|^{p^\ast-2}u_n\varphi \,dx \to \int_\Omega |u|^{p^\ast-2}u\varphi\,dx\quad \text{and}\quad
		\int_\Omega |v_n|^{q^\ast-2}v_n\psi\,dx \to \int_\Omega |v|^{q^\ast-2}v\psi\,dx,
		\]
		for any $(\varphi,\psi)\in W$. Then,  passing to the limit in \eqref{E'(un,vn)(phi,psi)}, with $(\varphi,\psi)=(u, v),\, (u, 0)$ and $(0, v)$  yields
		\begin{align}\label{E'(u,v)(u,v)}
			\int_\Omega\left(|\nabla u|^p+|\nabla v|^q\right)\,dx -\lambda \int_\Omega \nabla H(u,v)\cdot(u,v)\,dx + \mu \int_\Omega|u|^{p^{\ast}}\,dx - \int_\Omega|v|^{q^{\ast}}\,dx=0
		\end{align}
		and
		\begin{align}
			\int_\Omega|\nabla u|^p\,dx -\lambda \int_\Omega H_s(u,v)u\,dx + \mu \int_\Omega|u|^{p^{\ast}}\,dx=0,\label{E'(u,v)(u,0)}\\
			\int_\Omega|\nabla v|^q\,dx -\lambda\int_\Omega H_t(u,v)v\,dx - \int_\Omega|v|^{q^{\ast}}\,dx=0.\label{E'(u,v)(0,v)}
		\end{align}
		Let $\widetilde{u}_n = u_n - u$, $\widetilde{v}_n = v_n - v$. As in \cite{Boccardo-Murat}, it is easy to check that $\nabla u_n\to\nabla u$ almost everywhere in $\Omega$ for a further subsequence. Consequently,
		by the Brezis-Lieb lemma (see \cite{Brezis-Lieb-1983}) we have
		\[
		\int_\Omega|\nabla \widetilde{u}_n|^p\,dx =\int_\Omega|\nabla {u}_n|^p\,dx-\int_\Omega|\nabla u|^p\,dx+o(1),\quad
		\int_\Omega|\nabla \widetilde{v}_n|^q\,dx =\int_\Omega|\nabla {v}_n|^q\,dx-\int_\Omega|\nabla v|^q\,dx+o(1)
		\]
		\[
		\int_\Omega|\widetilde u_n|^{p^{\ast}}\,dx=\int_\Omega|u_n|^{p^{\ast}}\,dx-\int_\Omega|u|^{p^{\ast}}\,dx+o(1),\quad
		\int_\Omega|\widetilde v_n|^{q^{\ast}}\,dx=\int_\Omega|v_n|^{q^{\ast}}\,dx-\int_\Omega|v|^{q^{\ast}}\,dx+o(1).
		\]
		Thus, subtracting \eqref{E'(u,v)(u,0)} from \eqref{E'(wn)un} we get 
		\begin{align}\label{E'(wn)un tilde}
			\int_\Omega|\nabla \widetilde{u}_n|^p\,dx + \mu \int_\Omega|\widetilde{u}_n|^{p^{\ast}}\,dx=o(1),
		\end{align}
		which implies that $\widetilde{u}_n\to0$ in $W^{1,p}_0(\Omega)$,  that is,  ${u}_n\to u$ in $W^{1,p}_0(\Omega)$. Also, subtracting \eqref{E'(u,v)(0,v)} from \eqref{E'(wn)vn} we obtain
		\begin{align*}
			\int_\Omega|\nabla \widetilde{v}_n|^q\,dx  =\int_\Omega|\widetilde{v}_n|^{q^{\ast}}\,dx+o(1).
		\end{align*}
		Using \eqref{S def}  we see that
		\[
		\int_\Omega|\nabla \widetilde{v}_n|^q\,dx \geq S_q\left(\int_\Omega|\widetilde{v}_n|^{q^*}\,dx \right)^{q/q^*} = S_q\left(\int_\Omega|\nabla\widetilde{v}_n|^{q}\,dx +o(1)\right)^{q/q^*}.
		\]
		If $w_n$ does not have a strongly convergent subsequence, then $\|\nabla \widetilde{v}_n\|_q\geq \delta>0$ for some $\delta$. Thus
		\begin{align}\label{tilde vn S}
			\int_\Omega|\nabla \widetilde{v}_n|^q\,dx \geq S_q^{N/q}+o(1).
		\end{align}
		Now, recalling that $\{w_n\}_n$ is bounded,  we  multiply \eqref{E'(wn)un} by $1/p^{\ast}$,  multiply \eqref{E'(wn)vn} by $1/q^{\ast}$  and subtract both results from \eqref{E(wn)}, obtaining
		\begin{align*}
			\frac{1}{N} \int_\Omega\left(|\nabla u_n|^p+|\nabla v_n|^q\right)\,dx -\lambda
			\int_\Omega \left[H(u,v)-\frac{1}{p^\ast}H_u(u,v)u-\frac{1}{q^\ast}H_v(u,v)v \right]\,dx + o(1) = c.
		\end{align*}
		Multiplying \eqref{E'(u,v)(u,v)} by $1/N$ and subtracting it from this last inequality we obtain
		\begin{align*}
			c&=\frac{1}{N} \int_\Omega\left(|\nabla \widetilde{u}_n|^p+|\nabla \widetilde{v}_n|^q\right)\,dx 
			-\lambda
			\int_\Omega \left[H(u,v)-\frac{1}{p}H_u(u,v)u-\frac{1}{q}H_v(u,v)v \right]\,dx\\
			&\quad-\frac{\mu}{N} \int_\Omega|u|^{p^{\ast}}\,dx +\frac{1}{N}\int_\Omega|v|^{q^{\ast}}\,dx
			+ o(1).
		\end{align*}
		Since \eqref{E'(wn)un tilde} and \eqref{tilde vn S} hold we see that
		\begin{equation}\label{c above}
			\frac{1}{N} S_q^{N/q} -\lambda
			\int_\Omega \left[H(u,v)-\frac{1}{p}H_u(u,v)u-\frac{1}{q}H_v(u,v)v \right]\,dx- \frac{\mu}{N}\int_\Omega|u|^{p^{\ast}}\,dx \leq c .
		\end{equation}
		Using again the Sobolev inequality \eqref{S def}, \eqref{E'(wn)un tilde} and \eqref{unif bdd} we reach
		\[
		S_p\left(\int_\Omega|u|^{p^{\ast}}\,dx\right)^\frac{p}{p^{\ast}} \le
		\int_\Omega|\nabla u|^p\,dx
		=\lim_{n\to\infty}\int_\Omega|\nabla {u}_n|^p\,dx \le K_{\tilde\lambda}^p,
		\]
		which means that
		\[
		\frac{1}{N}\int_\Omega|u|^{p^{\ast}}\,dx \leq \frac{K_{\tilde\lambda}^{p^{\ast}}}{NS_p^\frac{p^{\ast}}{p}}=: d_{\tilde\lambda}.
		\]
		This, combined with condition \ref{AR} and \eqref{c above}, yields
		\begin{equation}\label{c absurd}
			\frac{1}{N} S_q^{N/q} -\mu d_{\tilde\lambda} \leq c.
		\end{equation}
		Therefore, for $c\le S_q^{N/q}/N$, whenever a $(PS)_c$ sequence $\{{w}_n\}$ does not have a strong convergent subsequence, the inequality \eqref{c absurd} must occur. This proves this lemma.
	\end{proof}
	
	\begin{lemma}\label{Lemma PS bounded}
		Assume that $H$ satisfies \ref{H basic} and \ref{H subcritic}. Let $K>0$ and $c<0$. Then there are $\tilde\lambda, \tilde\mu>0$, depending on $K$,
		such that for $\lambda\in(0,\tilde\lambda]$ and $\mu\in(0,\tilde\mu]$,  each  $(PS)_c$ sequence for $E$, $\{w_n\}$, satisfying
		\begin{equation}\label{wn bounded <1}
			\limsup_{n\to\infty}\|w_n\|\leq K
		\end{equation}
		has a subsequence that converges in $W$.
	\end{lemma}
	\begin{proof}
		For $c<0$, let $\{w_n\}_n $ be a $(PS)_c$ sequence for $E$ in $W$ satisfying \eqref{wn bounded <1}. Since $w_n= (u_n, v_n)$  is bounded, taking a subsequence, we may assume $(u_n, v_n) \rightharpoonup (u, v)$   in $W$, $u_n \to u$  in $L^s(\Omega)$ for $s \in [1, p^{\ast})$,  $v_n\to v$ in $L^s(\Omega)$ for $s \in [1, q^{\ast})$, and $u_n(x) \to u(x)$, $v_n(x) \to v(x)$ almost everywhere in $\Omega$. 
		Notice that
		\[
		\|(u,v)\|\leq\liminf_{n\to\infty}\|(u_n,v_n)\|\leq K.
		\]
		As in the proof of Lemma \ref{Lemma PS} (see \eqref{c above} and above) we see that, if $w_n$ does not have a strongly convergent subsequence, it holds
		\begin{equation}\label{c above 2}
			\frac{1}{N} S_q^{N/q} -\lambda
			\int_\Omega \left[H(u,v)-\frac{1}{p}H_u(u,v)u-\frac{1}{q}H_v(u,v)v \right]\,dx- \frac{\mu}{N}\int_\Omega|u|^{p^{\ast}}\,dx \leq c.
		\end{equation}
		Using again the Sobolev inequality \eqref{S def} and \eqref{wn bounded <1} we reach
		\[
		\frac{1}{N}\int_\Omega|u|^{p^{\ast}}\,dx \le
		\frac{1}{N}\left(\frac{1}{S_p}\int_\Omega|\nabla u|^p\,dx\right)^\frac{p^{\ast}}{p}
		\le \frac{1}{N}\left(\frac{K^p}{S_p}\right)^\frac{p^{\ast}}{p}=:C_1.
		\]
		From \ref{H subcritic} and \eqref{H MH} we also have
		\[
		\int_\Omega \left|H(u,v)-\frac{1}{p}H_u(u,v)u-\frac{1}{q}H_v(u,v)v \right|\,dx\leq {C}_2
		\]
		for some $C_2(K,N,p,q,\Omega)>0$ ($C_2$ also depend on the exponents from \ref{H subcritic} and \eqref{H MH}).
		Recalling that $c<0$, these inequalities combined with \eqref{c above 2} yield
		\begin{equation*}
			\frac{1}{N} S_q^{N/q} <\mu C_1 +\lambda C_2,
		\end{equation*}
		which is not possible for $\lambda$ and $\mu$ small. Therefore, for $\lambda$ and $\mu$ small enough, $\{w_n\}$ must have a convergent subsequence.
	\end{proof}
	
	\section{Proof of the main results}\label{Sec proofs}
	
	In this section we present the proofs of our results on the existence and multiplicity of solutions to the system \eqref{Sys}.
	
	\subsection{Proof of Theorem \ref{th 1 =1}}
	Here we assume that $\mu>0$, $k,m\geq1$ are such that
	\(
	\lambda_k=\cdots=\lambda_{k+m-1}<\lambda_{k+m},
	\)
	and $0<\lambda<\lambda_k$. We are going to apply Theorem \ref{Theorem 2.5 MP} with $c^*=c^*(\mu):=S_q^{N/q}/N-\mu d_{\lambda_k}$. Due to Lemma \ref{Lemma PS} we know that the functional $E=E_{\lambda,\mu}$ satisfies $(PS)_c$ condition for all $c<c^*$.
	
	For $\delta\in(0,\lambda_k)$ to be chosen later, we consider
	\begin{align}\label{epsilon choice}
		\lambda\in(\lambda_k-\delta,\lambda_k),\quad 0<\varepsilon<\min\left\{\frac{\delta}{\lambda_k-\lambda}-1, \frac{\lambda_{k+m}-\lambda_{k+m-1}}{\delta} \right\}
	\end{align}
	and denote
	\begin{align}\label{epsilon_lambda and tau}
		\varepsilon_\lambda=\varepsilon(\lambda_k-\lambda)\quad\text{and}\quad\tau:=\lambda_{k+m-1}+\varepsilon_\lambda. 
	\end{align}
	Since $\tau\in(\lambda_{k+m-1},\lambda_{k+m})$, due to Theorem \ref{lambdak}\ref{lambdak(iii)} we know that
	\[
	i(\mathcal{M}^+ \setminus \Psi_{{\tau}}) = k+m-1.
	\]
	Since $\mathcal{M}^+ \setminus \Psi_{{\tau}}$ is an open symmetric subset of $\mathcal M$, then it has
	a compact symmetric subset $A_0$ of index $k + m-1$ (see the proof of Proposition 3.1 in
	Degiovanni and Lancelotti \cite{De-Lan-2007}).
	Now, let us define
	\[
	B_0 = \Psi_{\lambda_{k}} \cup\left(\mathcal{M}\setminus\mathcal{M}^+\right).
	\]
	Note that $\mathcal{M} \setminus B_0 = \mathcal{M}^+ \setminus \Psi_{\lambda_{k}}$. Here we have two possibilities: $\lambda_1=\cdots=\lambda_k$ or $\lambda_{l-1}<\lambda_l=\cdots=\lambda_k$ for some $l\geq2$. In the first case we have
	\[
	i(\mathcal{M} \setminus B_0) = i(\mathcal{M}^+ \setminus \Psi_{\lambda_{1}})=i(\emptyset)=0.
	\]
	In the second case,  using Proposition \ref{lambdak}\ref{lambdak(iii)} again, we get
	\[
	i(\mathcal{M} \setminus B_0) = i(\mathcal{M}^+ \setminus \Psi_{\lambda_{l}}) = l-1\leq k-1.
	\]
	In order to apply Theorem \ref{Theorem 2.5 MP}, 
	for $R>\rho>0$ to be chosen, let
	\[
	A = \{ w_R : w \in A_0 \} 
	,\quad
	B = \{ w_\rho  : w \in B_0 \}
	\quad\text{and}\quad
	X = \{ w_t : w \in A_0, \, t \in [0,R] \}.
	\]
	For $w=(u,v) \in \mathcal{M}$ and $t\geq0$, recalling that $w_t=(t^\frac{1}{p}u,t^\frac{1}{q}v)$, we have
	\begin{align}\label{E(wt) M}
		E(w_t) &=I(w_t)- {\lambda}{ \int_\Omega H(w_t)\,dx}  + \frac{\mu}{p^{\ast}} \int_\Omega |t^\frac{1}{p}u|^{p^{\ast}}\,dx -\frac{1}{q^{\ast}} \int_\Omega |t^\frac{1}{q}v|^{q^{\ast}}\,dx\nonumber\\
		&= t \left( 1 - {\lambda}{ \int_\Omega H(u,v)\,dx} \right) + \frac{\mu t^{p^{\ast}/p}}{p^{\ast}} \int_\Omega |u|^{p^{\ast}}\,dx -\frac{t^{q^{\ast}/q}}{q^{\ast}} \int_\Omega |v|^{q^{\ast}}\,dx.
	\end{align}
	By the boundedness of $\mathcal{M}$ and the Sobolev inequality, we get 
	\begin{equation}\label{bound from M}
		\frac{1}{p^{\ast}} \int_\Omega |u|^{p^{\ast}}\,dx\leq c_1\quad\mbox{and}\quad\frac{1}{q^{\ast}} \int_\Omega |v|^{q^{\ast}}\,dx\leq c_1
	\end{equation}
	for some $c_1=c_1(N,p,q)>0$, for all $w=(u,v)\in\mathcal{M}$. 
	For $w\in B_0$, since $ \int_\Omega H(w)\,dx=1/\Psi(w)$ for $w\in\mathcal{M}^+$ and $ \int_\Omega H(w)\,dx\leq0$ for $w\in\mathcal{M}\setminus\mathcal{M}^+$, from \eqref{E(wt) M} and \eqref{bound from M} we get
	\begin{equation*}
		E(w_t) \geq
		\begin{cases}
			t \left( 1 - \frac{\lambda}{ \lambda_k} \right)  -t^{q^{\ast}/q} c_1,&\text{if} \,\,w\in\Psi_{\lambda_k}\\
			t   -t^{q^{\ast}/q} c_1,&\text{if} \,\, w\in\mathcal{M}\setminus\mathcal{M}^+.
		\end{cases}
	\end{equation*}
	Then we see that  there exists $\rho> 0$ small enough, independently of $\mu>0$, such that
	\[
	\inf_{w \in B} E(w)= \inf_{w \in B_0} E(w_\rho)> 0.
	\]
	Now we observe that, for $w \in A_0\subset\mathcal{M}^+ \setminus \Psi_{{\tau}}$, it holds $ \int_\Omega H(w)\,dx>1/\tau>1/\lambda_{k+m}$ (see \eqref{epsilon_lambda and tau}). For $\sigma=1/(2\lambda_{k+m})$, it follows from Lemma \ref{Lemma H}\ref{Lemma H p star} that there exists $C>0$ such that
	\[
	\frac{1}{\lambda_{k+m}}<\int_\Omega H(w)\,dx\leq \frac{1}{2\lambda_{k+m}}+C\left(\int_\Omega |v|^{q^*}\,dx\right)^{\tilde{q}/q^*}+C\left(\int_\Omega |v|^{q^*}\,dx\right)^{\bar{q}/q^*}. 
	\]
	Then, there is $c_2=c_2(N,p,\tilde{p},\bar{p},q,\tilde{q},\bar{q},\Omega,\lambda_{k+m})\in(0,1)$ such that
	\begin{equation}\label{v bdd below}
		\frac{1}{q^{\ast}} \int_\Omega |v|^{q^{\ast}}\,dx\geq 2c_2,
	\end{equation}
	whenever $(u,v)\in A_0$. Hence, using \eqref{E(wt) M}, \eqref{bound from M} and \eqref{v bdd below} we get
	\[
	E(w_t) \leq {t}  + \frac{\mu t^{p^{\ast}/p}}{p^{\ast}} \int_\Omega |u|^{p^{\ast}}\,dx 
	-\frac{ t^{q^{\ast}/q}}{q^{\ast}} \int_\Omega |v|^{q^{\ast}}\,dx
	\leq {t}  + \mu c_1  t^{p^{\ast}/p}-2c_2 t^{q^{\ast}/q},\quad \forall w\in A_0,\,\forall t\geq0.
	\]
	Considering $0<\mu<c_2/c_1$, since $p^\ast\leq q^\ast$ we obtain
	\[
	E(w_t)\leq 2{t}  -c_2{ t^{q^{\ast}/q}}\leq 0,\quad\forall w\in A_0,\, \forall t\geq R,
	\]
	if $R>\rho$ is sufficiently large. In particular, we get 
	\begin{align*}
		\sup_{w\in A}E(w)=\sup_{w\in A_0}E(w_R)\leq 0.
	\end{align*}
	It remains to estimate $E$ on $X$. Using again \eqref{E(wt) M}, \eqref{bound from M} and \eqref{v bdd below}, for $0<\mu<c_2/c_1$ we get
	\[
	E(w_t) \leq {t} \left( 1 - \frac{\lambda}{\tau} +\mu c_1 \right) -c_2t^{q^{\ast}/q},
	\quad\forall w\in A_0,\,\forall t\geq0.
	\]
	By the definition of $\tau$, given in \eqref{epsilon_lambda and tau}, and the choice of $\varepsilon$, given in \eqref{epsilon choice}, we see that
	$$
	{a}_k:= 1 - \frac{\lambda}{\tau} +\mu c_1 <\frac{{(1+\varepsilon)(\lambda_{k}-\lambda )}}{\lambda_{k}}+\mu c_1\,<\,\frac{2\delta}{\lambda_{k}}
	$$ 
	if $0<\mu<\delta/(c_1\lambda_k)$. Now we  choose $\delta:=\lambda_k\min\{S_q(c_2q^{\ast})^{(N-q)/N}/(3q);1\}$. From \eqref{epsilon choice}  we obtain
	\begin{align*}
		\max_{w\in A_0,\,t\geq0} E(w_t) &\leq \max_{t\geq0}\left({a}_kt-c_2{ t^{q^{\ast}/q}}\right)=\frac{(q{a}_k)^{N/q}}{N(c_2q^\ast)^{(N-q)/q}}<\frac{(2\delta q/\lambda_k)^{N/q}}{N(c_2q^\ast)^{(N-q)/q}}=:b_k\\
		&<\frac{(3q\delta/\lambda_k)^{N/q}}{N(c_2q^\ast)^{(N-q)/q}}=
		\frac{S_q^{N/q}}{N}.
	\end{align*}
	Defining $\mu_k=\min\{c_2/c_1;\delta/(c_1\lambda_k);({S_q^{N/q}}/{N}-b_k)/d_{\lambda_k}\}$, with $d_{\tilde\lambda}$ as in Lemma \ref{Lemma PS}, we reach 
	\[
	\sup_{w \in X} E(w)=\sup_{w \in A_0,\,t\in[0,R]} E(w_t) < \frac{S_q^{N/q}}{N}-\mu d_{\lambda_k}=c^\ast(\mu)
	\]
	for all $\mu\in(0,\mu_k)$ and $\lambda\in(\lambda_k-\delta,\lambda_k)$. Thus, Theorem \ref{Theorem 2.5 MP} ensures that $E$ has $m$ distinct pairs of critical points $\pm w_j$, $j=1,\cdots,m$, satisfying
	\[
	0<E(w_j)\leq\sup_{w \in X} E(w)\leq \frac{(q{a}_k)^{N/q}}{N(c_2q^\ast)^{(N-q)/q}}.
	\]
	Denoting $E_{\lambda,\mu}$ the functional and $\pm w_j^{\lambda,\mu}=\pm (u_j^{\lambda,\mu},v_j^{\lambda,\mu})$ these critical points, 
	this last inequality implies that $E_{\lambda,\mu}(w_j^{\lambda,\mu})\to0$ as $\lambda\to\lambda_k^-$ and $\mu\to0^+$. This means that
	\begin{align}\label{E(w l mu)}
		\frac{1}{p} \int_\Omega& |\nabla u_j^{\lambda,\mu}|^p\, dx+\frac{1}{q} \int_\Omega|\nabla v_j^{\lambda,\mu}|^q dx -\lambda\int_\Omega H(u_j^{\lambda,\mu},v_j^{\lambda,\mu})\,dx-\frac{1}{q^{\ast}}\int_\Omega|v_j^{\lambda,\mu}|^{q^{\ast}}\,dx \nonumber\\
		&=-\frac{\mu}{p^{\ast}}\int_\Omega|u_j^{\lambda,\mu}|^{p^{\ast}}\,dx+ o(1).
	\end{align}
	Since $\pm w_j^{\lambda,\mu}$, $j=1,\cdots,m$, are critical points for $E_{\lambda,\mu}$, we have
	\begin{align}\label{E'(w l mu)u0}
		0=\frac{1}{p}E_{\lambda,\mu}'(w_j^{\lambda,\mu})(u_j^{\lambda,\mu},0)=\frac{1}{p}\int_\Omega |\nabla u_j^{\lambda,\mu}|^p\, dx-\frac{\lambda}{p}\int_\Omega H_u(u_j^{\lambda,\mu},v_j^{\lambda,\mu})u_j^{\lambda,\mu}\,dx+\frac{\mu}{p}\int_\Omega|u_j^{\lambda,\mu}|^{p^{\ast}}\,dx
	\end{align}
	and
	\begin{align}\label{E'(w l mu)0v}
		0=\frac{1}{q}E_{\lambda,\mu}'(w_j^{\lambda,\mu})(0,v_j^{\lambda,\mu})=\frac{1}{q}\int_\Omega |\nabla v_j^{\lambda,\mu}|^q\, dx-\frac{\lambda}{q}\int_\Omega H_v(u_j^{\lambda,\mu},v_j^{\lambda,\mu})v_j^{\lambda,\mu}\,dx-\frac{1}{q}\int_\Omega|v_j^{\lambda,\mu}|^{q^{\ast}}\,dx.
	\end{align}
	Adding the equations given in \eqref{E'(w l mu)u0} and \eqref{E'(w l mu)0v} and subtracting the results from \eqref{E(w l mu)} we get
	\begin{align}\label{E(w l mu)+-}
		-\lambda\int_\Omega &\left[H(u_j^{\lambda,\mu},v_j^{\lambda,\mu})-\frac{1}{p}H_u(u_j^{\lambda,\mu},v_j^{\lambda,\mu})u_j^{\lambda,\mu}-\frac{1}{q}H_v(u_j^{\lambda,\mu},v_j^{\lambda,\mu})v_j^{\lambda,\mu}\right]\,dx +\left(\frac{1}{q}-\frac{1}{q^{\ast}}\right)\int_\Omega|v_j^{\lambda,\mu}|^{q^{\ast}}\,dx \nonumber\\
		&= \mu\left(\frac{1}{p}-\frac{1}{p^{\ast}}\right)\int_\Omega|u_j^{\lambda,\mu}|^{p^{\ast}}\,dx+o(1).
	\end{align}
	We also know, due to Lemma \ref{Lemma PS} (see \eqref{unif bdd}), that $\{ w_j^{\lambda,\mu}\}_{\lambda,\mu}$ is bounded in $W$, for $\lambda\in(\lambda_k-\delta,\lambda_k)$ and $\mu\in(0,\mu_k)$. So, 
	\[
	\mu\int_\Omega|u_j^{\lambda,\mu}|^{p^{\ast}}\,dx\to0\quad\text{as}\quad \mu\to0^+.
	\]
	Thus, using \ref{AR}, it follows from \eqref{E(w l mu)+-} that
	\[
	\int_\Omega|v_j^{\lambda,\mu}|^{q^{\ast}}\,dx\to0\quad\text{as}\quad\lambda\to\lambda_k^-\quad\text{and}\quad \mu\to0^+.
	\]
	Then, using again the boundedness of $\{ u_j^{\lambda,\mu}\}_{\lambda,\mu}$ in $W^{1,p}_0(\Omega)$, due to condition \ref{H ep Cep} we see that
	\[
	\int_\Omega H(u_j^{\lambda,\mu},v_j^{\lambda,\mu})\,dx\to0\quad\text{as}\quad\lambda\to\lambda_k^-\quad\text{and}\quad \mu\to0^+.
	\]
	Jointing these last three convergences with \eqref{E(w l mu)} we obtain $w_j^{\lambda,\mu}\to0$, as $\lambda\to\lambda_k^-$ and $\mu\to0^+.$
	This concludes the proof of Theorem \ref{th 1 =1}.

	\subsection{Proof of Theorem \ref{Th2 >1}}
	
	The proof of Theorem \ref{Th2 >1} is based on Corollary \ref{Cor-2.6}. For $\lambda,\,\mu>0$, and $H$ satisfying \ref{H basic}-\ref{AR}, with $\gamma>1$, from Lemma \ref{Lemma PS} we know that $E$ satisfies the $(PS)_c$ condition for all 
	\[c<\frac{1}{N}S_q^{N/q}-\mu d_\lambda=:c^*(\lambda,\mu).\]
	To get a compact set satisfying the assumptions of Corollary \ref{Cor-2.6}, 
	let us consider $\widetilde{H}(u,v)=|u|^\alpha|v|^\beta$ with $\alpha,\,\beta>1$ such that $\frac{\alpha}{p}+\frac{\beta}{q}=1$. Since $\widetilde{H}$ satisfies \ref{H basic}-\ref{AR}, with $\gamma=1$ in \ref{H scaling}, we can apply Theorem \ref{lambdak} to get a sequence  $\{\lambda_k\}$ of positive eigenvalues for problem \eqref{NE}. We have $\lambda_k\nearrow\infty$ and so, for each $m\geq1$ there is $k\geq m$ such that
	$\lambda_k<\lambda_{k+1}$. Let us assume $k=m$. 
	From Theorem \ref{lambdak}\ref{lambdak(iii)} we know that
	\[
	i(\mathcal{M}^+\setminus\Psi_{\lambda_{m+1}})=m.
	\]
	We recall that $\mathcal{M}^+=\{w\in W:I(w)=1\,\text{and}\, \int_\Omega\widetilde{H}(w)\,dx>0\}$.
	Since $\mathcal{M}^+ \setminus \Psi_{\lambda_{m+1}}$ is an open symmetric subset of $\mathcal M$, then it has
	a compact symmetric subset $A_0$ of index $m$ (see the proof of Proposition 3.1 in
	Degiovanni and Lancelotti \cite{De-Lan-2007}).
	In order to apply Corollary \ref{Cor-2.6}, 
	for $R>\rho>0$ to be chosen later, let
	\[
	A = \{ w_R : w \in A_0 \} 
	\quad\text{and}\quad
	X = \{ w_t : w \in A_0, \, t \in [0,R] \}.
	\]
	Now, for a function $H$ satisfying \ref{H basic}-\ref{AR}, with $\gamma>1$, we are going to estimate the functional $E$, associated with system \eqref{Sys}, on these sets $A$ and $X$.
	For $w=(u,v) \in \mathcal{M}$ and $t\geq0$, recalling that $w_t=(t^\frac{1}{p}u,t^\frac{1}{q}v)$, we have
	\begin{align}\label{E(wt) M th2}
		E(w_t) &=I(w_t)- {\lambda}{ \int_\Omega H(w_t)\,dx}  + \frac{\mu}{p^{\ast}} \int_\Omega |t^\frac{1}{p}u|^{p^{\ast}}\,dx -\frac{1}{q^{\ast}} \int_\Omega |t^\frac{1}{q}v|^{q^{\ast}}\,dx\nonumber\\
		&= t  - {\lambda}t^\gamma{ \int_\Omega H(u,v)\,dx}  + \frac{\mu t^{p^{\ast}/p}}{p^{\ast}} \int_\Omega |u|^{p^{\ast}}\,dx -\frac{t^{q^{\ast}/q}}{q^{\ast}} \int_\Omega |v|^{q^{\ast}}\,dx.
	\end{align}
	By the boundedness of $\mathcal{M}$ and the Sobolev inequality, using also Lemma \ref{Lemma H}\ref{Lemma H < I(w)}, we get 
	\begin{equation}\label{bound from M2}
		\frac{1}{p^{\ast}} \int_\Omega |u|^{p^{\ast}}\,dx\leq c_1,\quad\frac{1}{q^{\ast}} \int_\Omega |v|^{q^{\ast}}\,dx\leq c_1
		\quad\mbox{and}\quad\int_\Omega H(u,v)\,dx\leq c_1
	\end{equation}
	for some $c_1=c_1(N,p,q)>0$, for all $w=(u,v)\in\mathcal{M}$. 
	Hence,
	due to \eqref{E(wt) M th2} and \eqref{bound from M2}, since $\mu>0$, we see that
	\begin{equation*}
		E(w_t) \geq t  -\lambda c_1t^\gamma - c_1t^{q^{\ast}/q},\quad\forall w\in\mathcal{M}.
	\end{equation*}
	Since $1<\min\{\gamma;q^*/q\}$,   there exists $\rho> 0$ small enough, independently of $\mu>0$, such that
	\begin{equation}\label{min Mrho 2}
		\inf_{w \in \mathcal{M}_\rho} E(w)= \inf_{w \in \mathcal{M}} E(w_\rho)> 0.
	\end{equation}
	Now, let us estimate $E$ over $A$. Using \ref{H basic}, we have $\int_\Omega H(u,v)\,dx>0$ if $\int_\Omega \widetilde{H}(u,v)\,dx>0$.  Since $A_0\subset \mathcal{M}^+$ is compact, there is a positive constant $c_2$ such that
	\begin{equation}\label{H bounded in A0}
		\int_\Omega H(u,v)\,dx\geq c_2,\quad\forall w\in\ A_0.
	\end{equation}
	Using \eqref{bound from M2}, for $u$, and \eqref{H bounded in A0}, from Lemma \ref{Lemma H}\ref{Lemma H p star}  there exists $C>0$ such that
	\[
	c_2\leq\int_\Omega H(w)\,dx\leq \frac{c_2}{2}+C\left(\int_\Omega |v|^{q^*}\,dx\right)^{\tilde{q}/q^*}+C\left(\int_\Omega |v|^{q^*}\,dx\right)^{\bar{q}/q^*},\quad\forall w\in A_0.
	\]
	Then, there is $c_3=c_3(N,p,\tilde{p},\bar{p},q,\tilde{q},\bar{q},\Omega,c_2)\in(0,1)$ for which one, whenever $(u,v)\in A_0$, it holds
	\begin{equation}\label{v bdd below A0}
		\frac{1}{q^{\ast}} \int_\Omega |v|^{q^{\ast}}\,dx\geq 2c_3.
	\end{equation}
	Hence, using \eqref{E(wt) M th2}, \eqref{bound from M2} for $u$,  and \eqref{v bdd below A0}, since $1<p^\ast/p\leq q^\ast/q$ and $\lambda H\geq0$, we get
	\begin{align}\label{new eq 2}
		E(w_t) &\leq {t}  + \frac{\mu t^{p^{\ast}/p}}{p^{\ast}} \int_\Omega |u|^{p^{\ast}}\,dx 
		-\frac{ t^{q^{\ast}/q}}{q^{\ast}} \int_\Omega |v|^{q^{\ast}}\,dx
		\leq {t}  + \mu c_1  t^{p^{\ast}/p}-2c_3 t^{q^{\ast}/q}\nonumber\\
		&\leq (1 + \mu c_1){t}   -(2c_3 -\mu c_1)t^{q^{\ast}/q},
	\end{align}
	for all $w\in A_0$ and $t\geq0$.
	Considering $0<\mu<c_3/c_1$, we obtain
	\[
	E(w_t)\leq 2{t}  -c_3{ t^{q^{\ast}/q}}\leq 0,\quad\forall w\in A_0,\, \forall t\geq R,
	\]
	if $R>\rho$ is sufficiently large. In particular, we get 
	\begin{align}\label{sup in A th2}
		\sup_{w\in A}E(w)=\sup_{w\in A_0}E(w_R)\leq 0.
	\end{align}
	It remains to estimate $E$ on $X$. Using again \eqref{E(wt) M th2}, \eqref{bound from M2}, \eqref{H bounded in A0} and \eqref{v bdd below A0}, and arguing as in \eqref{new eq 2}, for $0<\mu<c_3/c_1$ we obtain
	\[
	E(w_t) \leq 2{t} -\lambda c_2t^\gamma -c_3t^{q^{\ast}/q}
	\leq 2t-\lambda c_2t^\gamma,
	\quad\forall w\in A_0,\,\forall t\geq0.
	\]
	Thus, we can choose $\Lambda^*_m>0$ sufficiently large such that
	\begin{align*}
		\max_{w\in A_0,\,t\geq0} E(w_t) &\leq \max_{t\geq0}\left(2t-\lambda c_2{ t^{\gamma}}\right)=\frac{2^{\gamma/(\gamma-1)}(\gamma-1)}{\gamma(\lambda\gamma c_2)^{1/(\gamma-1)}}<
		\frac{S_q^{N/q}}{2N},
	\end{align*}
	for all $\lambda\geq \Lambda^*_m$. 
	For each $\lambda\geq \Lambda^*_m$, we consider $d_{\lambda}$ as in Lemma \ref{Lemma PS} and fix $\mu_\lambda \in(0,c_3/c_1)$ such that  
	\[
	\sup_{w \in X} E(w)=\sup_{w \in A_0,\,t\in[0,R]} E(w_t) < \frac{S_q^{N/q}}{N}-\mu d_{\lambda}=c^\ast(\mu,\lambda)
	\]
	for all $\mu\in(0,\mu_\lambda)$. Thus, jointing this last estimate with \eqref{min Mrho 2} and \eqref{sup in A th2}, we conclude that Corollary \ref{Cor-2.6} can be applied. It ensures that, for each $\lambda\geq\Lambda^*_m$ and $\mu\in(0,\mu_\lambda)$, $E$ has $m$ distinct pairs of critical points $\pm w_j$, $j=1,\cdots,m$, satisfying
	\[
	0<E(w_j)\leq\sup_{w \in X} E(w).
	\]
	This concludes the proof of Theorem \ref{Th2 >1}.
	
	
	\subsection{Proof of Theorem \ref{Th3 <1}}
	
	
	The proof of this theorem is  based on Proposition \ref{critical point result}. 
	We will use a truncation of the functional $E$  inspired by  J. G. Azorero and I. P.  Alonso \cite{Azorero-Alonso-91}. We denote the functional associated with system \eqref{Sys} by
	\begin{equation*}
		E_{\lambda,\mu}(u,v)= I(u,v)-\lambda\int_\Omega H(u,v)\,dx+\frac{\mu}{p^{\ast}}\int_\Omega|u|^{p^{\ast}}\,dx-\frac{1}{q^{\ast}}\int_\Omega|v|^{q^{\ast}}\,dx
	\end{equation*}
	for
	\[
	I(u,v)=\frac{1}{p}\int_\Omega|\nabla u|^p\,dx+\frac{1}{q}\int_\Omega|\nabla v|^q\,dx.
	\]
	Since $H$ satisfies \ref{H basic}-\ref{H subcritic} and 
	\ref{H scaling} (with $\gamma<1$), from Lemma \ref{Lemma H}\ref{Lemma H < I(w)} we know that there exists $\tilde{c}_1>0$ such that
	\begin{equation*}
		\int_{\Omega } H(u,v)dx\leq \tilde{c}_1[I(u,v)]^\gamma,\quad\forall (u,v)\in W.
	\end{equation*}
	Also, from the Sobolev inequality, we get
	\[
	\frac{1}{q^*}\int_\Omega|v|^{q^{\ast}}\,dx \le
	\frac{1}{q^*}\left(\frac{1}{S_q}\int_\Omega|\nabla v|^q\,dx\right)^\frac{q^{\ast}}{q}
	\leq \tilde{c}_2[I(u,v)]^\frac{q^*}{q}.
	\]
	Thus, denoting $g_\lambda(t)=t- \tilde{c}_1\lambda t^\gamma-\tilde{c}_2t^\frac{q^*}{q}$, for $t\geq0$, since $\mu>0$ we have
	\begin{eqnarray*}
		E_{\lambda,\mu}(u,v)&\geq& I(u,v)- \tilde{c}_1\lambda[I(u,v)]^\gamma-\tilde{c}_2[I(u,v)]^\frac{q^*}{q}=g_\lambda(I(u,v)),\quad\forall (u,v)\in W.
	\end{eqnarray*}
	Recalling that $q^*/q>1>\gamma$, we see that is possible to find $\Lambda^*>0$ such that for each $\lambda\in(0,\Lambda^*)$ there exist $R_1(\lambda),R_2(\lambda)>0$ for which ones it holds
	\[
	g_\lambda(t)<0,\,\,\forall t\in[0,R_1(\lambda))\cup(R_2(\lambda),\infty)\qquad\mbox{and}\qquad g_\lambda(t)\geq0,\,\,\forall t\in[R_1(\lambda),R_2(\lambda)].
	\]
	Notice that if $\lambda\in(0,\Lambda^*)$ then $g_\lambda(t)\geq g_{\Lambda^*}(t)$ for all $t\in[R_1(\Lambda^*),R_2(\Lambda^*)]$. So, $R_1(\lambda)\leq R_1(\Lambda^*)$ and $R_2(\Lambda^*)\leq R_2(\lambda)$.
	Now we consider $\xi_\lambda:[0,\infty)\to[0,1]$ a smooth function satisfying 
	$$
	\xi_\lambda\equiv1\,\,\mbox{in}\,\, [0,R_1(\lambda)]\quad\mbox{and}\quad \xi_\lambda\equiv0\,\,\mbox{in}\,\,[R_2(\lambda),\infty).
	$$
	Finally, set the truncated functional $\tilde{E}_{\lambda,\mu}:W\to \mathbb{R}$ given by
	$$
	\tilde{E}_{\lambda,\mu}(w)=\xi_\lambda(I(w)){E}_{\lambda,\mu}(w).
	$$
	We observe that if $\tilde{E}_{\lambda,\mu}(w)<0$ then $I(w)\in(0,R_1(\lambda))$ and, due to the continuity of $I$, $I(\tilde{w})\in(0,R_1(\lambda))$ for $\tilde{w}$ in a neighborhood of $w$ in $W$. So
	\begin{equation}\label{truncation<0}
		\tilde{E}_{\lambda,\mu}(w)={E}_{\lambda,\mu}(w)<0\quad\mbox{and}\quad \tilde{E}_{\lambda,\mu}'(w)={E}_{\lambda,\mu}'(w).
	\end{equation}
	In particular, critical points for $\tilde{E}_{\lambda,\mu}$ at negative levels are also critical points for ${E}_{\lambda,\mu}$ at the same levels.
	Therefore, we will apply Proposition \ref{critical point result} to show that $\tilde{E}_{\lambda,\mu}$ has a sequence of critical points at negative levels. Our first step is to verify that  $\tilde{E}_{\lambda,\mu}$ satisfies the $(PS)_c$ condition for $c<0$, if $\lambda\in(0,\Lambda^*)$ with $\Lambda^*$ sufficiently small. Let $(u_n)$ be a sequence in $E$ such that 
	$$
	\tilde{E}_{\lambda,\mu}(w_n)\to c<0\quad\text{and}\quad \tilde{E}_{\lambda,\mu}'(w_n)\to 0.
	$$
	Since $\tilde{E}_{\lambda,\mu}(w_n)<0$ for large $n$, we have $I(w_n)\in(0,R_1(\Lambda^*))$, which means that $(w_n)$ is uniformly bounded for $\lambda\in(0,\Lambda^*)$. Up to a subsequence, $w_n\rightharpoonup w$ in $W$.  The uniform boundedness of  $I(w_n)\in(0,R_1(\Lambda^*))$   ensures that $I(u)\leq R_1(\Lambda^*)$, independently of $\lambda\in(0,\Lambda^*)$. 
	Moreover, \eqref{truncation<0} implies that $(w_n)$ is also a $(PS)_c$ sequence for ${E}_{\lambda,\mu}$. Hence, Lemma \ref{Lemma PS bounded} ensures that
	$w_n\to w$ in $W$ for some subsequence, if $\mu\in(0,\mu^*)$ and  $\lambda\in(0,\Lambda^*)$ for some $\mu^*>0$ and 
	$\Lambda^*>0$ small enough. Therefore,  $\tilde{E}_{\lambda,\mu}$ satisfies the $(PS)_c$ condition for all $c<0$ if $\lambda\in(0,\Lambda^*)$ and $\mu\in(0,\mu^*)$. Now, for fixed parameters $\lambda\in(0,\Lambda^*)$ and $\mu\in(0,\mu^*)$, our second step is to estimate the numbers $c_k$, given in equation \eqref{ck gamma<1}, which are 
	\begin{equation*}
		c_k=\inf_{M\in \mathcal{F}_k}\sup_{u\in M}\tilde{E}_{\lambda,\mu}(u)
		\quad\text{for}\quad
		\mathcal{F}_k = \{M \subset W \backslash\{0\}: M\,\,\text{is symmetric and}\,\, i(M) \geq k\}.
	\end{equation*}
	The boundedness of $\tilde{E}_{\lambda,\mu}$ from below guarantees that $c_k>-\infty$, for all $k\in\mathbb{N}$. We need to show that $c_k<0$ for large $k$. As before, we consider $\widetilde{H}(u,v)=|u|^\alpha|v|^\beta$ for $\alpha,\beta>1$ satisfying $\frac{\alpha}{p}+\frac{\beta}{q}=1$ and denote $(\lambda_k)$ the sequence of positive eigenvalues for problem \eqref{NE} given in Theorem \ref{lambdak}. Since $\lambda_k\nearrow\infty$, for arbitrarily large $m\in\mathbb{N}$ there is $k\geq m$ such that $\lambda_k<\lambda_{k+1}$. From Theorem \ref{lambdak}\ref{lambdak(iii)} we have
	\[
	i(\mathcal{M}^+\setminus\Psi_{\lambda_{k+1}})=k.
	\]
	We recall that $\mathcal{M}^+=\{w\in W:I(w)=1\,\text{and}\, \int_\Omega\widetilde{H}(w)\,dx>0\}$ and $\Psi(w)=1/\int_\Omega\widetilde{H}(w)\,dx$.
	Since $\mathcal{M}^+ \setminus \Psi_{\lambda_{k+1}}$ is an open symmetric subset of $\mathcal M$,  it has
	a compact symmetric subset $M$ of index $k$ (see the proof of Proposition 3.1 in
	Degiovanni and Lancelotti \cite{De-Lan-2007}).
	For $t > 0$, let $M_t = \{w_t : w \in M \}$. Since the mapping $M \to M_t$ , $w \mapsto w_t$ is an odd
	homeomorphism, it follows that
	\[i(M_t)= i(M)= k,\quad\forall t>0,\]
	which implies that $M\in \mathcal{F}_k $. We also have $I(w_t)=t<R_1(\lambda)$ for all $w\in\mathcal{M}$, if $t>0$ is small and so  $E_{\lambda,\mu}(w_t)= \tilde{E}_{\lambda,\mu}(w_t)$. Therefore, we are going to estimate  $E_{\lambda,\mu}(w_t)$ for $w$ in the compact $M$.
	The Sobolev inequality implies that
	\[
	\frac{1}{p^*}\int_{\Omega}|u|^{p^*}dx\leq c_1,\quad\forall (u,v)\in W.
	\]
	Since \ref{H basic} implies that $\int_\Omega H(u,v)\,dx>0$ in $\mathcal{M}^+$, due to the compactness of $M$, there exists $c_2>0$ such 
	\[
	\int_\Omega H(u,v)\,dx\geq c_2,\quad\forall (u,v)\in M.
	\]
	Hence, it follows from Lemma \ref{Lemma H}\ref{Lemma H p star} (see \eqref{v bdd below A0}) the existence of $c_3\in(0,1)$ such that
	\[
	\frac{1}{q^*}\int_{\Omega}|v|^{q^*}dx\geq 2c_3,\quad\forall (u,v)\in M.
	\]
	Notice that the constants $c_2$ and $c_3$ do not depend on $\mu$ nor $\lambda$. Recalling that   $p^*/p,q^*/q>1>\gamma$, for all $w\in M$ we obtain
	\begin{align*}
		E_{\lambda,\mu}(w_t)&=t-\lambda t^\gamma\int_\Omega\widetilde{H}(w)\,dx + \frac{\mu t^\frac{p^*}{p}}{p^*}\int_{\Omega}|u|^{p^*}dx
		-\frac{t^\frac{q^*}{q}}{q^*}\int_{\Omega}|v|^{q^*}dx\\
		&\leq (1+\mu c_1)t-\lambda c_2t^\gamma-(2c_3-\mu c_1)t^\frac{q^*}{q}
		\,\leq \, 2t-\lambda c_2t^\gamma<0
	\end{align*}
	if $\mu^*<c_3/c_1$ and $t>0$ is sufficiently small.  
	So, fixing a $t\in(0,R_1(\lambda))$ small we have $\sup_{w\in M}E_{\lambda,\mu}(w_{t})<0$.
	Thus, $\tilde{E}_{\lambda,\mu}(w)=E_{\lambda,\mu}(w)$ for all $w\in M_{t}$ and
	\begin{equation*}
		c_k=\inf_{M\in \mathcal{F}_k}\sup_{w\in M}\tilde{E}_{\lambda,\mu}(w)\leq \sup_{u\in M_{t}}\tilde{E}_{\lambda,\mu}(w)=\sup_{u\in M_{t}}{E}_{\lambda,\mu}(w)<0.
	\end{equation*}
	Thus, Proposition \ref{critical point result} can be applied to get a sequence $\{w_k^{\lambda,\mu}\}_k$ of critical points for $\tilde{E}_{\lambda,\mu}$ at negative energy levels. By definition of $\tilde{E}_{\lambda,\mu}$ we see that
	\(
	\tilde{E}_{\lambda,\mu}(w_k^{\lambda,\mu})<0
	\)
	implies that $I(w_k^{\lambda,\mu})<R_1(\lambda).$ Moreover, from the definition of $R_1(\lambda)$, since
	\[
	g_\lambda(\lambda)=\lambda- \tilde{c}_1\lambda^{\gamma+1}-\tilde{c}_2\lambda^\frac{q^*}{q}>0
	\]
	for small $\lambda>0$, we have $R_1(\lambda)<\lambda$. In particular, $R_1(\lambda)\to0^+$ as $\lambda\to0^+$.
	This ensures that 
	\[
	I(w_k^{\lambda,\mu})=\frac{1}{p}\int_\Omega|\nabla u_k^{\lambda,\mu}|^p\,dx+\frac{1}{q}\int_\Omega|\nabla v_k^{\lambda,\mu}|^q\,dx\to0\quad\text{as}\quad\lambda\to0
	\]
	and concludes this proof.
	
	\section*{Acknowledgements}
	
	Elisandra Gloss would like to thank Conselho Nacional de Desenvolvimento Científico e Tecnol\'ogico (CNPq) for the support. In particular, for the grants 201454/2024-6 and 443594/2023-6.

	\section*{Conflict of interest statement}
	
	The authors have no relevant financial or non-financial interests to disclose.

\end{document}